\documentclass{amsart}

\usepackage{amssymb}

\newtheorem{theorem}{Theorem}[section]

\theoremstyle{definition}
\newtheorem{definition}[theorem]{Definition}

\newtheorem{proposition}[theorem]{Proposition}

\theoremstyle{remark}
\newtheorem{remark}[theorem]{Remark}

\numberwithin{equation}{section}

\usepackage{indentfirst,enumerate,cite,amssymb,amsfonts,amsmath,amsthm,mathrsfs,dsfont}

\usepackage[colorlinks=true,linkcolor=blue,citecolor=blue]{hyperref}

\allowdisplaybreaks

\begin{document}
	
	\title[Global Solutions for Systems on closed manifolds]{Global Solutions for Systems of Strongly Invariant Operators on Closed Manifolds}
	
	
	\author[A. Kirilov]{Alexandre Kirilov}
	\address{
		Departamento de Matem\'atica 
		Universidade Federal do Paran\'a  
		Caixa Postal 19096, CEP 81530-090, Curitiba 
		Brasil}
	\email{akirilov@ufpr.br}
	\thanks{This study was financed in part by the Coordenação de Aperfeiçoamento de Pessoal de Nível Superior - Brasil (CAPES) - Finance Code 001. The first and second authors were supported in part by CNPq - Brasil (grants 316850/2021-7 and 423458/2021-3).}
	
	\author[W. A. A. de Moraes]{Wagner Augusto Almeida de Moraes}
	\address{
		Departamento de Matem\'atica  
		Universidade Federal do Paran\'a  
		Caixa Postal 19096\\  CEP 81530-090, Curitiba, Paran\'a 
		Brasil}
	\email{wagnermoraes@ufpr.br}
	
	\author[P. M. Tokoro]{Pedro Meyer Tokoro}
	\address{
		Programa de P\'os-Gradua\c c\~ao em Matem\'atica  
		Universidade Federal do Paran\'a  
		Caixa Postal 19096\\  CEP 81530-090, Curitiba, Paran\'a  
		Brasil}
	\email{pedro.tokoro@ufpr.br}
	
	
	
	
	
	\begin{abstract}
		We study the global hypoellipticity and solvability of strongly invariant operators and systems of strongly invariant operators on closed manifolds. Our approach is based on the Fourier analysis induced by an elliptic pseudo-differential operator, which provides a spectral decomposition of $L^2(M)$ into finite-dimensional eigenspaces. This framework allows us to characterize these global properties through asymptotic estimates on the matrix symbols of the operators. Additionally, for systems of normal strongly invariant operators, we derive an explicit solution formula and establish sufficient conditions for global hypoellipticity and solvability in terms of their eigenvalues. 
	\end{abstract}

	\maketitle

	\section{Introduction}

	In this paper, we investigate the solvability and regularity of solution of systems of operators defined on closed manifolds (i.e. compact smooth manifolds without boundary). Our approach is based on Fourier analysis associated with an elliptic pseudo-differential operator on $M$, which provides a spectral decomposition of $L^2(M)$ into finite-dimensional eigenspaces. This framework, originally introduced by Seeley \cite{See67}, offers the appropriate setting for studying the regularity and solvability of strongly invariant operators through their spectral behavior.  
	
	The concept of strongly invariant operators, as introduced by Delgado and Ruzhansky \cite{DR2018}, considers linear operators $P:C^\infty(M)\to L^2(M)$ that extend continuously to $\mathscr{D}'(M)$ and commute with a fixed elliptic operator. This class of operators includes, for instance, left-invariant vector fields on compact Lie groups and certain classes of (pseudo)differential operators and Fourier multipliers.  
	
	With these tools at our disposal, we characterize the global properties of a single strongly invariant operator $P$ through asymptotic estimates on its matrix symbol $\sigma_P$. Our approach recovers the result of Kirilov and de Moraes \cite{KM2020}, who provided a full characterization of the global hypoellipticity of a single strongly invariant operator. Furthermore, inspired by the work of Araújo \cite{Araujo2019}, we establish necessary and sufficient conditions for the closedness of the range os $P$.
	
	Next, we extend these results to systems $\mathbb{P}=(P_1,\ldots,P_n)$ of strongly invariant operators, establishing necessary and sufficient conditions for their global hypoellipticity and solvability. The main challenge in this setting arises from the interplay between multiple operators, particularly in the absence of commutativity. Addressing this requires precise spectral estimates on the matrix symbol $\sigma_\mathbb{P}=(\sigma_{P_1},\ldots,\sigma_{P_n})$.  
	
	Moreover, under the additional assumption that the operators $P_j$ are normal, we derive an explicit formula for the solution of the system $\mathbb{P}=(P_1,\ldots,P_n)$ and reformulate our previous results to obtain conditions for global hypoellipticity and solvability based on the eigenvalues of the operators $P_j$. This case recovers the study of global properties of systems of left-invariant vector fields on compact Lie groups. Specifically, if $G$ is a compact Lie group with Lie algebra $\mathfrak{g}$, every left-invariant vector field $X\in\mathfrak{g}$ is a skew-symmetric operator (hence normal) that commutes with the Laplace-Beltrami operator $\Delta_G$, which is elliptic. The global hypoellipticity and solvability of vector fields on compact Lie groups were characterized by conditions on their eigenvalues, as proved in \cite{DKP2024arxiv,KMR20,KMR20b}.  
	
	Finally, we show that assuming commutativity, i.e., $[P_j, P_k] = P_jP_k - P_kP_j = 0$ for all $j,k = 1,\dots,n$, leads to a significant simplification of the previously obtained results. In this setting, our results recover classical theorems on the torus, including those from \cite{AKM19,AM2021,BP1999_jmaa,Be1999}, among others.

	\section{Fourier Analysis on Closed Smooth Manifolds} \label{overview}
	
	Let $M$ be a $d$-dimensional closed smooth manifold equipped with a positive measure $dx$. We denote by $L^2(M)$ the space of square-integrable complex-valued functions on $M$ with respect to this measure and by $H^s(M)$ the standard Sobolev space of order $s\in\mathbb{R}$ defined on $M$. The spaces of smooth functions and distributions are characterized by the following identities:
	\begin{equation*}\label{sobolev-dx}
		C^\infty(M) = \bigcap_{s \in \mathbb{R}} H^{s}(M) \quad \text{and} \quad {\mathscr{D}'}(M) = \bigcup_{s \in \mathbb{R}} H^{s}(M).
	\end{equation*}
	
	These spaces are endowed with the projective and inductive limit topologies, respectively, making $C^\infty(M)$ an FS space and ${\mathscr{D}'}(M)$ a DFS space. Notice that closed subspaces of FS and DFS spaces remain FS and DFS, respectively. Furthermore, the Closed Graph Theorem and the Open Mapping Theorem hold in these spaces. For further details, see \cite{Komatsu1967,kothe_TVS}.
	
	To establish a framework for Fourier analysis on $M$, following the construction proposed by Delgado and Ruzhansky in \cite{DR2018}, we consider a classical positive definite elliptic pseudo-differential operator $E$ of order $\nu > 0$ on $M$. The eigenvalues of $E$, counted without multiplicity, form a sequence $\{ \lambda_k \}_{k \in \mathbb{N}_0}$ satisfying
	$$
	0 = \lambda_0 < \lambda_1 < \lambda_2 < \dots \to \infty.
	$$
	
	Due to the ellipticity of $E$, each eigenvalue $\lambda_k$ corresponds to a finite-dimensional eigenspace $E_{\lambda_k} \subset C^\infty(M)$ of dimension $d_k$, satisfying the summability condition
	\begin{equation*}
		\sum_{k\in\mathbb{N}_0} d_k(1+\lambda_k)^{-q} < \infty \quad \text{for any } \quad q > d/\nu.
	\end{equation*}
	In particular, there exists a constant $C>0$ such that
	\begin{equation}\label{asymp}
		d_k \leq C(1+\lambda_k)^ {d/ \nu}, \quad k\in\mathbb{N}_0.
	\end{equation}
	
	An orthonormal basis for $L^2(M)$, consisting of smooth eigenfunctions of $E$, can be chosen as
	\[
	\mathcal{B} =\{ \varphi_\ell^k \mid 1 \leq \ell \leq d_k,  k \in \mathbb{N}_0 \},
	\]
	where, for each $k \in \mathbb{N}_0$, the finite subset $\mathcal{B}_k=\{\varphi_1^k,\dots,\varphi_{d_k}^k\}$ forms an orthonormal basis for the eigenspace $E_{\lambda_k}$. Moreover, $L^2(M)$ admits the following decomposition:
	$$
	L^2(M) = \bigoplus_{k \in \mathbb{N}_0} E_{\lambda_k}.
	$$
	
	For a function $u \in L^2(M)$, we can represent it as
	\begin{equation}\label{fourier-series-u-in-L2}
		u(x) = \sum_{k\in\mathbb{N}_0} \langle \widehat{u}(k), \overline{\varphi^k(x)} \rangle_{\mathbb{C}^{d_k}}
		= \sum_{k=0}^\infty \sum_{\ell=1}^{d_k} \widehat{u}(k)_\ell \varphi_\ell^k(x),
	\end{equation}
	where $\varphi^k(x) \doteq \big(\varphi^k_1(x), \dots, \varphi^k_{d_k}(x)\big)$, $\widehat{u}(k) \doteq \big(\widehat{u}(k)_1, \dots, \widehat{u}(k)_{d_k}\big)$, and the Fourier coefficients with respect to the orthonormal basis $\mathcal{B}_k$ are given by
	\begin{equation*}
		\widehat{u}(k)_\ell = \int_M u(x)\overline{\varphi^k_\ell(x)} \, dx, \quad 1 \leq \ell \leq d_k, \quad k \in \mathbb{N}_0.
	\end{equation*}
	
	Given $u \in L^2(M)$, the Hilbert-Schmidt norm of $\widehat{u}(k)$ is given by
	\begin{equation*}
		\|\widehat{u}(k)\|=\left(\sum_{\ell=1}^{d_k}|\widehat{u}(k)_\ell|^2\right)^{1/2},
	\end{equation*}
	and we obtain the Plancherel formula:
	\begin{equation*}
		\|u\|_{L^2(M)}=\sum_{k=0}^{\infty}\sum_{\ell=1}^{d_k}|\widehat{u}(k)_\ell|^2=\sum_{k=0}^{\infty}\|\widehat{u}(k)\|^2.
	\end{equation*}
	
	The Fourier coefficients of a distribution $u \in {\mathscr{D}'}(M)$ are given by $\widehat{u}(k)_\ell \doteq u(\overline{\varphi_\ell^k}),$  
	and $u$ can be expanded in a series as in \eqref{fourier-series-u-in-L2}. Moreover, a distribution $u \in {\mathscr{D}'}(M)$ belongs to the Sobolev space $H^s(M)$ if and only if 
	$$
	\|u\|_{H^s} \doteq \left[\sum_{k=0}^\infty \sum_{\ell=1}^{d_k} (1+\lambda_k)^{2s/ \nu} |\widehat{u}(k)_\ell|^2 \right]^{1/2}< \infty.
	$$ 
	
	Smooth functions on $M$ are characterized by the rapid decay of their Fourier coefficients: $f \in C^\infty(M)$ if and only if, for every $N \in \mathbb{N}$, there exists a constant $C_N > 0$ such that  
	\begin{equation}\label{fourier_coef}
		\|\widehat{f}(k)\| \leq C_N(1+\lambda_k)^{-N}, \quad \text{for all } k \in \mathbb{N}.
	\end{equation}  
	By duality, distributions on $M$ are characterized by the slow growth of their Fourier coefficients: $u \in {\mathscr{D}'}(M)$ if and only if there exist $N \in \mathbb{N}$ and $C > 0$ such that  
	\begin{equation}\label{fourier_coef-distr}
		\|\widehat{u}(k)\| \leq C(1+\lambda_k)^N, \quad \text{for all } k \in \mathbb{N}.
	\end{equation}  
	
	\begin{definition}
		We say that a linear operator $P:C^\infty(M)\to L^2(M)$ is \textit{strongly invariant relative to} $E$ when the domain of its adjoint operator $P^*$ contains $C^\infty(M)$, it extends to a continuous linear operator $P:\mathscr{D}'(M)\to\mathscr{D}'(M)$, and if $P$ satisfies one of the following equivalent conditions:
		\begin{itemize}
			\item[(a)] $P(E_{\lambda_k})\subset E_{\lambda_k}$, for each $k\in\mathbb{N}_0$;
			\item[(b)] $PE\varphi_\ell^k=EP\varphi_\ell^k$, for each $k\in\mathbb{N}_0$ and $\ell=1,\dots,d_k$;
			\item[(c)] For each $k\in\mathbb{N}_0$, there exists a matrix $\sigma_P(k)\in \mathbb{C}^{d_k\times d_k}(\mathbb{C})$ such that 
			$$
			\widehat{P\varphi_\ell^N}(k)_j=\sigma_P(k)_{j,\ell}\delta_{N,k},  \quad \text{for all } \varphi_\ell^N \in \mathcal{B}_k;
			$$
			\item[(d)] For each $k\in\mathbb{N}_0$, there exists a matrix $\sigma_P(k)\in M_{d_k}(\mathbb{C})$ such that 
			$$
			\widehat{Pf}(k)=\sigma(k)\widehat{f}(k), \quad f\in C^\infty(M);
			$$
			\item[(e)] The commutator $[E,P]=EP-PE$ vanishes on $L^2(M)$.
		\end{itemize}
	\end{definition}		
	
	Moreover, the matrices $\sigma_P(k)$ in conditions (c) and (d) coincide, and the sequence $\sigma_P \doteq \{\sigma_P(k)\mid k\in \mathbb{N}_0 \}$ is called the matrix symbol of $P$.	
	
	\begin{remark}
		A linear operator $P$ satisfying the first four conditions above is called an invariant operator (or a Fourier multiplier) relative to $E$. Since condition (e) is clearly stronger than the first four, the assumption that $P$ extends to $\mathscr{D}'(M)$ is included to ensure its equivalence with the other conditions. This justifies the term strongly above. The proof of these equivalences is provided in Section 4 of \cite{DR2018}.
	\end{remark}
	
	\begin{definition}
		Let $P:C^\infty(M)\to L^2(M)$ be a strongly invariant operator relative to $E$. We say that its matrix symbol  $\sigma_P = \{\sigma_P(k) \mid k\in \mathbb{N}_0 \}$ \textit{has moderate growth} if there exist $N\in\mathbb{N}$ and $C>0$ such that
		\begin{equation}\label{order}
			\|\sigma_P(k)\|\leq C(1+\lambda_k)^{N/\nu},\ \forall k\in\mathbb{N}_0.
		\end{equation}
		In this case, we define the order of $\sigma_P$ by
		\begin{equation*}
			\text{ord}(\sigma_P)=\inf\{N\in\mathbb{R} \mid \text{condition  \eqref{order} holds}\}.
		\end{equation*}
		When the symbol of $P$ has moderate growth, we say that the order of $P$ is the order of its symbol.
	\end{definition}
	
	\begin{remark}
		Throughout this paper, we assume that a classical positive definite elliptic pseudo-differential operator $E$ of order $\nu > 0$ on $M$ is fixed. Furthermore, whenever we refer to a \textbf{strongly invariant operator} $P$, it should be understood that we are considering an operator $P:C^\infty(M)\to L^2(M)$ that is strongly invariant with respect to $E$, with a matrix symbol $\sigma_P = \{\sigma_P(k) \mid k\in \mathbb{N}_0 \}$ of moderate growth.
	\end{remark}

	\section{Global Properties of Strongly Invariant Operators}
	
	In this section, we study the global regularity and solvability of a single strongly invariant operator. While the global hypoellipticity of such operators was characterized by Kirilov and de Moraes in \cite{KM2020}, for global solvability, we adapt some ideas and results obtained by Araújo in \cite{Araujo2019}, proving that the closedness of the range of an operator is equivalent to a weaker version of global regularity.
	
	We begin by defining the three global properties studied in this paper:
	
	\begin{definition}\label{global-prop-operators}
		Let $P:\mathscr{D}'(M)\to \mathscr{D}'(M)$ be a linear operator. We say that:
		\begin{enumerate}
			\item $P$ is \textit{globally hypoelliptic} if the conditions $u\in\mathscr{D}'(M)$ and $Pu\in C^\infty(M)$ imply that $u\in C^\infty(M);$
			\item $P$ is \textit{almost globally hypoelliptic} if the conditions $u\in\mathscr{D}'(M)$ and $Pu\in C^\infty(M)$ imply that there exists $v\in C^\infty(M)$ such that $Pu=Pv;$ 
			\item $P$ is \textit{globally solvable} if the induced map $P:C^\infty(M)\to C^\infty(M)$ has a closed image.
		\end{enumerate}
	\end{definition}
	
	We begin by recalling a result established in \cite{KM2020}.
	
	\begin{theorem}\label{KM}
		A strongly invariant operator $P$ is globally hypoelliptic if and only if there exist constants $C,\gamma >0$ such that
		\begin{equation*}
			k\geq \gamma\ \Rightarrow\ \|\sigma_P(k)\varphi\|\geq C(1+\lambda_k)^\gamma,\ \forall \varphi\in E_{\lambda_k},\ \|\varphi\|_{L^2}=1.
		\end{equation*}
	\end{theorem}
	
	The next results in this section aim to establish the conditions required for a strongly invariant operator to be globally solvable. To achieve this, we will show the equivalence between global solvability and almost global hypoellipticity, following the approach proposed in \cite{Araujo2019}.
	
	\begin{proposition}\label{est_1}
		Let $P$ be a strongly invariant operator. Suppose that there exist $t\in\mathbb{R}$ and $C>0$ such that
		\begin{equation}\label{6.1}
			\|\sigma_P(k)\varphi\|\geq C(1+\lambda_k)^{t/\nu},\ \forall \varphi\in\ker\sigma_P(k)^\perp,\ \|\varphi\|_{L^2}=1,\ k\in\mathbb{N}_0.
		\end{equation}
		Then, for all $u\in\mathscr{D}'(M)$ , there exists $v\in\ker P$ such that
		\begin{equation}
			\forall s\in\mathbb{R},\ Pu\in H^s(M)\ \Rightarrow\ u-v\in H^{s+t}(M).
		\end{equation}
		In this case, the following estimate holds:
		\begin{equation*}
			\|u-v\|_{H^{s+t}}\leq C^{-1}\|Pu\|_{H^s}.
		\end{equation*}
	\end{proposition}
	
	\begin{proof}
		Let $u\in\mathscr{D}'(M)$. For each $k\in\mathbb{N}_0$, we decompose
		$\widehat{u}(k) = \widehat{v}(k) + \widehat{w}(k),$ where $\widehat{v}(k) \in \ker \sigma_P(k)$ and $\widehat{w}(k) \in \ker \sigma_P(k)^\perp$. 		
		Applying \eqref{6.1} to $\widehat{w}(k)$, we obtain
		\begin{align*}
			(1+\lambda_k)^{t/\nu} \|\widehat{u}(k) - \widehat{v}(k)\| 
			= (1+\lambda_k)^{t/\nu} \|\widehat{w}(k)\| 
			& \leq C^{-1} \|\sigma_P(k) \widehat{w}(k)\| \\
			&= C^{-1} \|\sigma_P(k) \widehat{u}(k)\|.
		\end{align*}
		
		Squaring both sides and multiplying by $(1+\lambda_k)^{2s/\nu}$, we obtain
		\begin{equation*}
			(1+\lambda_k)^{2(t+s)/\nu} \|\widehat{u}(k) - \widehat{v}(k)\|^2 
			\leq C^{-2} (1+\lambda_k)^{2s/\nu} \|\sigma_P(k) \widehat{u}(k)\|^2.
		\end{equation*}
		
		Summing over $k\in\mathbb{N}_0$ and taking the square root, we conclude that
		\begin{equation*}
			\|u-v\|_{H^{s+t}} \leq C^{-1} \|Pu\|_{H^s}.
		\end{equation*}
		
		In particular, if $Pu\in H^s(M)$, then $u-v\in H^{s+t}(M)$, as desired.
	\end{proof}
		
	\begin{proposition}\label{agh_1}
		Suppose that, for all $s\in\mathbb{R}$ and $C>0$, condition \eqref{6.1} does not hold. Then, there exists $u\in\mathscr{D}'(M)$ such that $Pu\in C^\infty(M)$ but $u+v\notin C^\infty(M)$ for all $v\in\ker P$.
	\end{proposition}
	
	\begin{proof}
		Let $\rho>d = \dim(M)$ and fix $s\in\mathbb{R}$. We will show that there exists $u\in H^s(M)$ such that $Pu\in C^\infty(M)$ but $u+v\notin H^{s+\rho/2}(M)$ for all $v\in\ker P$.
		
		By hypothesis, for each $\ell\in\mathbb{N}$, there exist $k_\ell\in\mathbb{N}_0$ and $\varphi_\ell\in\ker\sigma_P(k_\ell)$ with $\|\varphi_\ell\|_{L^2}=1$ such that $\|\sigma_P(k_\ell)\varphi_\ell\|<2^{-\ell}(1+\lambda_{k_\ell})^{-\ell/\nu}$.
		
		Without loss of generality, we may assume that the sequence $\{\lambda_{k_\ell}\}_{\ell\in\mathbb{N}}$ is increasing. Define $u\in\mathscr{D}'(M)$ by
		\[
		\widehat{u}(k)=\begin{cases}
			(1+\lambda_{k_\ell})^{-(s+\rho/2)/\nu}\varphi_\ell,&\text{if } k=k_\ell \text{ for some } \ell\in\mathbb{N},\\
			0,&\text{otherwise}.
		\end{cases}
		\]
		
		Since, by Weyl's asymptotic formula, we have $
			\sum_{k=0}^\infty (1+\lambda_k)^{-\rho/\nu}<\infty$, it follows that for all $t\in\mathbb{R}$,
		\begin{align*}
			\sum_{k=0}^\infty (1+\lambda_k)^{2(s+t)/\nu}\|\widehat{u}(k)\|^2 
			&= \sum_{\ell=1}^\infty (1+\lambda_{k_\ell})^{2(s+t)/\nu} \frac{\|\varphi_\ell\|^2}{(1+\lambda_{k_\ell})^{2(s+\rho/2)/\nu}}\\
			&= \sum_{\ell=1}^{\infty}(1+\lambda_{k_\ell})^{2(t-\rho/2)/\nu}.
		\end{align*}
		
		Therefore, $u\in H^s(M)\setminus H^{s+\rho/2}(M)$. On the other hand, we have
		\begin{align*}
			\sum_{k=0}^\infty (1+\lambda_k)^{2t/\nu}\|\widehat{u}(k)\|^2 
			\leq \sum_{\ell=1}^\infty 2^{-2\ell}(1+\lambda_{k_\ell})^{-2(t-s-\rho/2-\ell)/\nu} < \infty,\quad \forall t\in\mathbb{R}.
		\end{align*}
		
		Thus, $Pu\in C^\infty(M)$. Furthermore, if $v\in\ker P$, the Pythagorean theorem gives
		\begin{equation*}
			\|\widehat{u}(k)+\widehat{v}(k)\|^2 = \|\widehat{u}(k)\|^2+\|\widehat{v}(k)\|^2 \geq \|\widehat{u}(k)\|^2,\quad \forall k\in\mathbb{N}_0.
		\end{equation*}
		This implies that if $u\notin H^s(M)$, then $u+v\notin H^s(M)$, completing the proof.
	\end{proof}
	
	As a direct consequence of the previous proposition, we obtain the following:
	
	\begin{theorem}\label{agh_est}
		A strongly invariant operator $P$ is almost globally hypoelliptic if and only if there exist $s\in\mathbb{R}$ and $C>0$ such that condition (\ref{6.1}) holds.
	\end{theorem}
	
	\begin{theorem}\label{equiv_gs_agh}
		A strongly invariant operator $P$ is globally solvable if and only if it is almost globally hypoelliptic.
	\end{theorem}
	
	\begin{proof}
		\noindent($\Leftarrow$) Suppose that $P$ is almost globally hypoelliptic.
		Let $\{u_\ell\}_{\ell\in\mathbb{N}}$ be a sequence in $C^\infty(M)$ and let $f\in C^\infty(M)$ be such that $Pu_\ell\to f$ in $C^\infty(M)$. Then,
		\begin{equation*}
			\lim\limits_{\ell\to\infty}\sigma_P(k)\widehat{u}_\ell(k) = \widehat{f}(k), \text{ for all } k\in\mathbb{N}_0.
		\end{equation*}
		
		Since each $\sigma_P(k):E_{\lambda_k}\to E_{\lambda_k}$ has a closed image (as $E_{\lambda_k}$ is finite-dimensional), there exists $\widehat{u}(k)$ such that $\sigma_P(k)\widehat{u}(k)=\widehat{f}(k)$. Furthermore, we can choose $\widehat{u}(k)$ such that $\widehat{u}(k) \in \ker\sigma_P(k)^\perp$.  
		
		By Theorem \ref{agh_est}, there exist constants $C>0$ and $t\in\mathbb{R}$ such that
		\begin{equation*}
			\|\widehat{u}(k)\|\leq C^{-1}(1+\lambda_k)^{-t/\nu}\|\widehat{f}(k)\|.
		\end{equation*}

		Since $f\in C^\infty(M)$, this bound implies that $u\in C^\infty(M)$. Moreover, we have $Pu=f$, which shows that $f\in P(C^\infty(M))$. Hence, $P$ is globally solvable.  
		
		\medskip\noindent
		($\Rightarrow$) Suppose that $P$ is globally solvable, i.e., $P(C^\infty(M))$ is closed.
		We first show that the space $\mathcal{F}=\{u\in C^\infty(M)\,:\, \widehat{u}(k)\in\ker\sigma_P(k)^\perp,\ \forall k\in\mathbb{N}_0\}$	is closed in $C^\infty(M)$. 
		
		Indeed, if $\{u_\ell\}_{\ell\in\mathbb{N}}$ is a sequence such that $u_\ell\to u$ in $C^\infty(M)$, then $\widehat{u}_\ell(k)\to\widehat{u}(k)$ for all $k\in\mathbb{N}_0$. Since each $\ker\sigma_P(k)^\perp$ is closed in $E_{\lambda_k}$, it follows that $u\in\mathcal{F}$.  
		
		Since $P:\mathscr{D}'(M)\to\mathscr{D}'(M)$ is continuous, the Closed Graph Theorem implies that the restriction $P:C^\infty(M)\to C^\infty(M)$ is also continuous. This, in turn, implies that the induced operator $P:\mathcal{F}\to P(C^\infty(M))$ is a continuous linear bijection between FS spaces. By the Open Mapping Theorem, its inverse  $P^{-1}: P(C^\infty(M))\to\mathcal{F}$ is also continuous.  
		
		By the definition of the projective limit topology, this continuity means that for each $t>0$, there exist $s>0$ and $C>0$ such that
		\begin{equation*}
			\|P^{-1}v\|_{H^t}\leq C\|v\|_{H^s},\quad \forall v\in P(C^\infty(M)).
		\end{equation*}
		Equivalently, for all $u\in\mathcal{F}$,
		\begin{equation*}
			\|u\|_{H^t}\leq C\|Pu\|_{H^s}.
		\end{equation*}
		
		In particular, for each $k\in\mathbb{N}_0$ and $\varphi\in\ker\sigma_P(k)^\perp$, we have $\varphi\in\mathcal{F}$, so
		\begin{equation*}
			(1+\lambda_k)^t\|\varphi\|=\|\varphi\|_{H^t}\leq C\|\sigma_P(k)\varphi\|_{H^s}=C(1+\lambda_k)^s\|\sigma_P(k)\varphi\|.
		\end{equation*}
		This implies that
		\begin{equation*}
			\|\sigma_P(k)\varphi\|\geq C^{-1}(1+\lambda_k)^{t-s}\|\varphi\|,\quad \forall \varphi\in\ker\sigma_P(k)^\perp, \quad \forall k\in\mathbb{N}_0.
		\end{equation*}
		
		By Theorem \ref{agh_est}, this establishes that $P$ is almost globally hypoelliptic.
	\end{proof}

	We now observe that condition (\ref{6.1}) holds for some $t\in\mathbb{R}$ and $C>0$ if and only if, for every $k\in\mathbb{N}_0$, the number
	\begin{equation*}
		m(\sigma_P(k)) \doteq \inf\{\|\sigma_P(k)\varphi\| \mid \varphi\in\ker\sigma_P(k)^\perp, \|\varphi\|_{L^2} =1\}  \geq C(1+\lambda_k)^{t/\nu}.
	\end{equation*}
	
	\begin{theorem}
		A strongly invariant operator $P$ is globally solvable if and only if there exist positive constants $C,\gamma$ such that
		\begin{equation*}
			k\geq \gamma \Rightarrow m(\sigma_P(k))\geq C(1+\lambda_k)^\gamma.
		\end{equation*}
	\end{theorem}
	
	It is clear that global hypoellipticity implies almost global hypoellipticity. Consequently, every globally hypoelliptic operator is also globally solvable. The condition in the theorem above is, as expected, a weaker version of the condition in Theorem \ref{KM}.
	
	\begin{remark}
	An interesting refinement occurs when $P$ is a normal operator. In this case, each $\sigma_P(k)$ is a normal matrix, meaning that for each eigenspace $E_{\lambda_k}$, there exists an orthonormal basis in which $\sigma_P(k)$ is diagonal:
	\[
	\sigma_P(k) = \text{diag}(\mu_1(k),\dots,\mu_{d_k}(k)), \quad \mu_\ell(k)\in\mathbb{C}, \quad \ell=1,\dots,d_k.
	\]
	
	Here, $m(\sigma_P(k))$ corresponds to the smallest nonzero eigenvalue of $\sigma_P(k)$ in norm, i.e., $m(\sigma_P(k)) = |\mu_{\ell_0}(k)|$ for some index $\ell_0$. This reformulation shows that the global solvability of $P$ can be directly characterized by the asymptotic behavior of its nonzero eigenvalues. This observation aligns with Theorem 3.5 of \cite{KMR20b}, which addresses the global properties of left-invariant vector fields on compact Lie groups.
	\end{remark}
	
\section{Global Properties for Systems of Strongly Invariant Operators}

	In this section, we extend the notions of global regularity and solvability to systems of strongly invariant operators, motivated by results from \cite{AFR2022,AFR2024}. 
	
	\begin{definition}
		Let $\mathbb{P}=(P_1,\dots,P_n):\mathscr{D}'(M)\to (\mathscr{D}'(M))^n$ be a system of linear operators. We say that:
		\begin{enumerate}
			\item $\mathbb{P}$ is \textit{globally hypoelliptic} if $u\in\mathscr{D}'(M)$ and $P_j u\in C^\infty(M)$, for $1\leq j \leq n$, imply that $u\in C^\infty(M).$ Equivalently, if $u\in\mathscr{D}'(M)$ and $\mathbb{P}u\in (C^\infty(M))^n$ then $u\in C^\infty(M)$.  
			
			\item $\mathbb{P}$ is \textit{almost globally hypoelliptic} if $u\in\mathscr{D}'(M)$ and $\mathbb{P}u\in (C^\infty(M))^n$ imply the existence of $v\in C^\infty(M)$ such that $\mathbb{P}u=\mathbb{P}v.$ 
			
			\item $\mathbb{P}$ is \textit{globally solvable} if the induced map $\mathbb{P}:C^\infty(M)\to (C^\infty(M))^n$ has a closed image.
		\end{enumerate}
	\end{definition}

	Let $\mathbb{P}=(P_1,\dots,P_n)$ be a system of strongly invariant operators, where each $P_j:\mathscr{D}'(M)\to \mathscr{D}'(M)$ for $j=1,\dots,n$. For simplicity, we will denote the matrix symbol $\sigma_{P_j}(k)$ of each operator $P_j$ by $\sigma_j(k)$ for all $k\in\mathbb{N}$ and $j\in\{1,\dots,n\}$. Accordingly, the matrix symbol of the system $\mathbb{P}$ is given by  
	\begin{equation*}
		\sigma_{\mathbb{P}}(k) = (\sigma_1(k), \ldots, \sigma_n(k)), \quad k\in\mathbb{N}.
	\end{equation*}
	
	Moreover, for each $k\in\mathbb{N}_0$, we define $\mathbb{P}(k):E_{\lambda_k}\to (E_{\lambda_k})^n$ as the restriction of $\mathbb{P}$ to the eigenspace $E_{\lambda_k}$.

	Since $\ker\mathbb{P} = \bigcap_{j=1}^n\ker P_j,$  then we have $\ker \mathbb{P}(k) = \bigcap_{j=1}^n\ker\sigma_j(k)$ for each $k\in\mathbb{N}_0$. By the invariance property of the operators, for each $j=1,\dots,n$ and $k\in\mathbb{N}_0$, we have $P_j(E_{\lambda_k})\subset E_{\lambda_k}$, so the restriction $P_j|_{E_{\lambda_k}}$ acts as a linear operator on the finite-dimensional space $E_{\lambda_k}$. Fixing a basis $\mathcal{B}_k = \{\varphi^k_1, \dots, \varphi^k_{d_k}\}$ for $E_{\lambda_k}$, we represent the matrix of $P_j|_{E_{\lambda_k}}$ in this basis as  
\begin{equation*}
	\sigma_j(k)=\begin{bmatrix}
		\sigma_j(k)_{1,1}&\cdots & \sigma_j(k)_{1,d_k}\\
		\vdots & \ddots & \vdots\\
		\sigma_j(k)_{d_k,1}&\cdots& \sigma_j(k)_{d_k,d_k}
	\end{bmatrix}.
\end{equation*}

Now, given $u\in\mathscr{D}'(M)$ and $f_j\in C^\infty(M)$ such that $P_j u = f_j$ for all $j=1,\dots,n$, we expand these functions in terms of the eigenbasis as  
\begin{equation*}
	u = \sum_{k\in\mathbb{N}_0}\sum_{\ell=1}^{d_k}\widehat{u}(k)_\ell\varphi_\ell^k,  
	\quad \text{and} \quad  
	f_j = \sum_{k\in\mathbb{N}_0}\sum_{\ell=1}^{d_k}\widehat{f_j}(k)_\ell\varphi_\ell^k, \quad j=1,\dots,n.
\end{equation*}

	For each $k\in\mathbb{N}_0$, solving for $\widehat{u}(k)$ reduces to the overdetermined system of linear equations:
	\begin{equation}\label{system}
		\sigma(k) \cdot \widehat{u}(k) =
		\begin{bmatrix}
			\sigma_1(k)\\ \vdots \\ \sigma_n(k)
		\end{bmatrix} \cdot \widehat{u}(k)
		=
		\begin{bmatrix} 
			\widehat{f_1}(k)\\ \vdots \\ \widehat{f_n}(k)
		\end{bmatrix} = F(k),
	\end{equation}
	where  $\widehat{u}(k) = [\widehat{u}(k)_1 \ \cdots \ \widehat{u}(k)_{d_k}]^T$ and $\widehat{f_j}(k) =
		[ \widehat{f_j}(k)_1 \ \cdots \ \widehat{f_j}(k)_{d_k}]^T$.

	For this system to have a unique solution, it is necessary that  
	\begin{equation*}
		\bigcap_{j=1}^{n}\ker\sigma_j(k)=\{0\}.
	\end{equation*}
	
	Indeed, given functions $u,v$ satisfying $P_ju = P_jv = f_j$ for all $j=1,\dots,n$, we observe that their difference satisfies  $	u-v\in\bigcap_{j=1}^n\ker\sigma_j(k),$ for all $k\in\mathbb{N}_0.$ 
	Thus, uniqueness of the solution is equivalent to requiring that the intersection of the kernels is trivial. We now define the set
	\begin{equation*}
		\mathcal{Z}=\left\{k\in\mathbb{N}_0\,:\, \ker\mathbb{P}(k)\neq \{0\}\right\}.
	\end{equation*}
	The next result provides a necessary condition for the global hypoellipticity of $\mathbb{P}$.
	
	\begin{proposition}\label{prop_Z}
		If the system $\mathbb{P}$ is globally hypoelliptic, then $\mathcal{Z}$ is finite.
	\end{proposition}	
	\begin{proof}
		Suppose, by contradiction, that $\mathcal{Z}$ is infinite. For each $k\in\mathcal{Z}$, choose a nonzero element $v(k)\in\bigcap_{j=1}^n\ker\sigma_j(k)$ with $\|v(k)\|=1$. Then, the distribution  
		$u=\sum_{k\in\mathcal{Z}}v(k) \in \mathscr{D}'(M)\setminus C^\infty(M)$ and satisfies $P_j u = 0$ for $j=1,\dots,n$.
	\end{proof}

	Considering the 2-norm on each product space $(E_{\lambda_k})^n$, we establish a new characterization of the global hypoellipticity of the system $\mathbb{P}$. Our proof follows a similar approach to Theorem 3.3 of \cite{KM2020} and adapts techniques from Proposition 5.2 of \cite{AFR2024}, extending both results. More precisely, we provide necessary and sufficient conditions for the injectivity of each $\mathbb{P}(k)$ (except for finitely many indices) and, in such cases, obtain explicit control over the norm of its inverse.
	
	\begin{theorem}\label{gh_est}
		The system $\mathbb{P}$ is globally hypoelliptic if and only if there exist constants $C,\gamma>0$ such that
		\begin{equation}\label{GH-system-condition}
			k\geq\gamma \Rightarrow \|\mathbb{P}(k)\varphi\| = \left( \sum_{j=1}^n \|\sigma_j(k)\varphi\|^2 \right)^{1/2} \!\!\! \geq C(1+\lambda_k)^{-\gamma},\  \forall \varphi\in E_{\lambda_k}, \|\varphi\|=1.
		\end{equation}
	\end{theorem}
	
	\begin{proof}
		Suppose first that the estimate holds, and let $u\in\mathscr{D}'(M)$ be such that $P_ju=f_j\in C^\infty(M)$ for all $j=1,\dots,n$. Then, for $k\geq\gamma$, we obtain
		\begin{equation*}
			C^2(1+\lambda_k)^{-2\gamma}\|\widehat{u}(k)\|^2\leq \sum_{j=1}^n\|\sigma_j(k)\widehat{u}(k)\|^2 = \sum_{j=1}^n\|\widehat{f_j}(k)\|^2.
		\end{equation*}
		
		Since each $f_j$ is smooth, there exists a constant $K>0$ such that, for all $N\in\mathbb{N}$, we have  
		\begin{equation*}
			\|\widehat{f_j}(k)\|\leq K(1+\lambda_k)^{-N-\gamma}, \ j=1,\dots,n.
		\end{equation*}
		
		Consequently,  
		\begin{equation*}
			\|\widehat{u}(k)\|^2\leq nK^2C^{-2}(1+\lambda_k)^{2\gamma}(1+\lambda_k)^{-2N-2\gamma} = nK^2C^{-2}(1+\lambda_k)^{-2N},
		\end{equation*}
		which implies that $u$ is smooth. Thus, $\mathbb{P}$ is globally hypoelliptic.  
		
		Conversely, suppose that the estimate does not hold. Then, for each $N\in\mathbb{N}$, there exist $k_N\geq N$ and $\varphi_N\in E_{\lambda_{k_N}}$ with $\|\varphi_N\|=1$ such that
		\begin{equation*}
			\left(\sum_{j=1}^n\|\sigma_j(k_N)\varphi_N\|^2\right)^{1/2} < 2^{-N}(1+\lambda_{k_N})^{-N}.
		\end{equation*}
		
		Note that  $u=\displaystyle\sum_{N\in\mathbb{N}}\varphi_N \in\mathscr{D}'(M)\setminus C^\infty(M)$. 
		Furthermore, for $k = k_N$, we obtain  
		\begin{equation*}
			\|\sigma_j(k)\widehat{u}(k)\|\leq \left(\sum_{j=1}^n\|\sigma_j(k)\widehat{u}(k)\|^2\right)^{1/2}< 2^{-N}(1+\lambda_{k_N})^{-N},
		\end{equation*}
		while for all other values of $k$, we have $\sigma_j(k)\widehat{u}(k)=0$. Therefore $P_j u \in C^\infty(M)$ for$j=1,\dots,n$, contradicting the global hypoellipticity of $\mathbb{P}$.
	\end{proof}
	
	Following the same reasoning as in the previous results, we now establish a characterization for the global solvability of systems of operators.
	
	\begin{theorem}\label{agh_gs_system}
		Let $\mathbb{P}$ be a system of strongly invariant operators. The following statements are equivalent:
		\begin{itemize}
			\item[(a)] There exist constants $C,\rho>0$ such that
			\begin{equation*}
				\left(\sum_{j=1}^{n}\|\sigma_{j}(k)\varphi\|^2\right)^{1/2} \!\!\! \geq C(1+\lambda_k)^{-\rho}\|\varphi\|, \text{ for } \varphi \in \ker\mathbb{P}(k)^\perp, \text{and } k\in\mathbb{N}_0.
			\end{equation*}
			\item[(b)] $\mathbb{P}$ is almost globally hypoelliptic.
			\item[(c)] $\mathbb{P}$ is globally solvable.
		\end{itemize}
	\end{theorem}
	
	\begin{proof}
		The equivalence between (b) and (c) follows from the same arguments used in the proof of Theorem \ref{equiv_gs_agh}, noting that $(C^\infty(M))^n$ and $(\mathscr{D}'(M))^n$ remain FS and DFS spaces, respectively. The equivalence between (a) and (b) is analogous to Propositions \ref{est_1} and \ref{agh_1}.
	\end{proof}

\section{Systems of Normal Operators}

	Let $\mathbb{P} = (P_1,\dots,P_n)$ be a system of strongly invariant operators. By assuming that each $P_j$ is normal, it follows that its restriction $P_j|_{E_{\lambda_k}}$ is a normal operator on a finite-dimensional vector space. Consequently, each symbol $\sigma_j(k)$ is a diagonalizable matrix. Thus, for each $j=1,\dots,n$ and $k\in\mathbb{N}_0$, there exists a unitary matrix  
	\begin{equation*}
		Q_j(k)=\begin{bmatrix}
			q_j(k)_{1,1}&\cdots&q_j(k)_{1,d_k}\\
			\vdots&\ddots&\vdots\\
			q_j(k)_{d_k,1}&\cdots&q_j(k)_{d_k,d_k}
		\end{bmatrix}
	\end{equation*}
	such that  
	\begin{equation*}
		\sigma_j(k) = Q_j(k)^*  
		\begin{bmatrix}
			\mu_j(k)_1& \cdots& 0\\
			\vdots&\ddots&\vdots\\
			0&\cdots& \mu_j(k)_{d_k}
		\end{bmatrix}  
		Q_j(k),
	\end{equation*}
	where $\mu_j(k)_1,\dots,\mu_j(k)_{d_k}$ are the eigenvalues of $\sigma_j(k)$.  
	
	For each $j=1,\dots,n$, we define  
	\begin{equation}\label{psi}
		\psi_{\ell}^{j,k} \doteq Q_j(k)\varphi_{\ell}^{k},\quad \ell=1,\dots,d_k.
	\end{equation}  
	
	Then, $\mathcal{B}_{j,k} = \{\psi_\ell^{j,k}\}_{\ell=1}^{d_k}$ forms a basis of $E_{\lambda_k}$ in which $\sigma_j(k)$ is diagonal. Moreover, $Q_j(k)$ represents the change-of-basis matrix from $\mathcal{B}_{k}$ to $\mathcal{B}_{j,k}$.  
	
	For each $j=1,\dots,n$ and $k\in\mathbb{N}_0$, we obtain the linear system  
	\begin{equation*}
		Q_j(k)^* \text{diag}(\mu_j(k)_1,\dots,\mu_j(k)_{d_k}) Q_j(k) \widehat{u}(k) = \widehat{f_j}(k).
	\end{equation*}
	
	Defining  
	\begin{equation}\label{QUW}
		Q_j(k)\widehat{u}(k) = W_j(k) =  
		\begin{bmatrix}
			w_j(k)_1& \cdots& w_j(k)_{d_k}
		\end{bmatrix}^T,
	\end{equation}
	and  
	\begin{equation}\label{QGF}
		Q_j(k)\widehat{f_j}(k) = G_j(k) =  
		\begin{bmatrix}
			g_j(k)_1& \cdots& g_j(k)_{d_k}
		\end{bmatrix}^T,
	\end{equation}
	we obtain the diagonal system  
	\begin{equation}\label{muwg}
		\mu_j(k)_\ell w_j(k)_\ell = g_j(k)_\ell,\quad \forall j=1,\dots,n,\ k\in\mathbb{N}_0,\ \ell=1,\dots,d_k.
	\end{equation}
	
	Since $U(k) = Q_j(k)^* W_j(k)$, we have  
	\begin{equation*}
		\begin{bmatrix}
			\widehat{u}(k)_1\\ \vdots \\ \widehat{u}(k)_{d_k}
		\end{bmatrix} =  
		\begin{bmatrix}
			\overline{q_j(k)_{1,1}} & \cdots& \overline{q_j(k)_{d_k,1}}\\
			\vdots&\ddots&\vdots\\
			\overline{q_j(k)_{1,d_k}} & \cdots & \overline{q_j(k)_{d_k,d_k}}
		\end{bmatrix}  
		\begin{bmatrix}
			w_j(k)_1\\ \vdots \\ w_j(k)_{d_k}
		\end{bmatrix}.
	\end{equation*}
	Thus,  
	\begin{equation}\label{uqw}
		\widehat{u}(k)_\ell = \sum_{m=1}^{d_k} \overline{q_j(k)_{m,\ell}} w_j(k)_{m},\quad\ell=1,\dots,d_k.
	\end{equation}
	
	Since the coefficients $g_j(k)_\ell$ belong to a smooth function $g_j$ for each $j=1,\dots,n$, assume that $k\in\mathbb{N}_0\setminus\mathcal{Z}$, meaning that $\bigcap_{j=1}^n\ker\sigma_j(k)=\{0\}$. Then, choosing any $j\in\{1,\dots,n\}$, we write  
	\begin{equation*}
		\widehat{u}(k)_\ell = \sum_{m=1}^{d_k} \overline{q_{j}(k)_{m,\ell}}w_{j}(k)_{m},\quad\ell=1,\dots,d_k,
	\end{equation*}
	as in (\ref{uqw}). If $\mu_{j}(k)_m\neq 0$, then by (\ref{muwg}) we have  
	\begin{equation*}
		w_{j}(k)_m = \dfrac{g_{j}(k)_m}{\mu_{j}(k)_m}.
	\end{equation*}
	
	For each $m\in\{1,\dots,d_k\}$ such that $\mu_{j}(k)_m=0$, we have $\psi_{m}^{j,k}\in\ker\sigma_{j}(k)$, where $\psi_{m}^{j,k}$ is given by (\ref{psi}). Then, we can choose $j_m\in\{1,\dots,n\}$ such that  
	\begin{equation*}
		\ker\sigma_{j_m}(k)\cap \text{span}\{\psi_{m}^{j,k}\}=\{0\}.
	\end{equation*}
	
	Such a choice of $j_m$ is possible due to the assumption that $k\notin\mathcal{Z}$. Otherwise, we would have  
	\begin{equation*}
		\bigcap_{j=1}^n\ker\sigma_j(k)\supset \text{span}\{\psi_{m}^{j,k}\} \neq \{0\}.
	\end{equation*}
	
	Since $\ker\sigma_{j_m}(k)$ is spanned by the $\psi_{\ell}^{j_m,k}$ for which $\mu_{j_m}(k)_\ell=0$, it follows from (\ref{QUW}) that  
	\begin{equation*}
		W_j(k) = Q_j(k) Q_{j_m}(k)^* W_{j_m}(k).
	\end{equation*}
	
	Denoting  
	\begin{equation*}
		Q_j(k)Q_{j_m}(k)^* =  
		\begin{bmatrix}
			c_{j_m}(k)_{1,1} & \cdots & c_{j_m}(k)_{1,d_k}\\
			\vdots & \ddots & \vdots\\
			c_{j_m}(k)_{d_k,1} & \cdots & c_{j_m}(k)_{d_k,d_k}
		\end{bmatrix},
	\end{equation*}
	we can express  
	\begin{equation*}
		w_j(k)_m = \sum_{\ell=1}^{d_k} c_{j_m}(k)_{m,\ell} w_{j_m}(k)_\ell.
	\end{equation*}
	
	Notice that $Q_{j_m}(k)Q_j(k)^*$ represents the change-of-basis matrix from $\mathcal{B}_{j,k}$ to $\mathcal{B}_{j_m,k}$. Therefore, we have
	\begin{equation*}
		Q_{j_m}(k)Q_j(k)^*=
		\begin{bmatrix}
			\overline{c_{j_m}(k)_{1,1}} & \cdots & \overline{c_{j_m}(k)_{d_k,1}}\\
			\vdots & \ddots & \vdots\\
			\overline{c_{j_m}(k)_{1,d_k}} & \cdots & \overline{c_{j_m}(k)_{d_k,d_k}}
		\end{bmatrix}.
	\end{equation*}
	
	In the basis $\mathcal{B}_{j,k}$, the vector $\psi_m^{j,k}$ is given by  
	\begin{equation*}
		\psi_{m}^{j,k}=\begin{bmatrix} 0 & \cdots & 0 & 1 & 0 & \cdots & 0 \end{bmatrix}_{\mathcal{B}_{j,k}}^T,
	\end{equation*}
	where the entry 1 appears only in the $m$-th row. Transforming this vector to the basis $\mathcal{B}_{j_m,k}$ using the change-of-basis matrix, we obtain  
	\begin{equation*}
		\psi_m^{j,k} = \sum_{s=1}^{d_k}\overline{c_{j_m}(k)_{m,s}}\psi_{s}^{j_m,k}.
	\end{equation*}
	
	Since the vectors $\psi_s^{j_m,k}$, with indices $s\in\{1\leq s\leq d_k \mid \mu_{j_m}(k)_s=0\},$ 
	generate $\ker\sigma_{j_m}(k)$, and we have chosen $j_m$ such that $		\ker\sigma_{j_m}(k) \cap \text{span}\{\psi_{m}^{j,k}\} = \{0\},$ it follows that  $\overline{c_{j_m}(k)_{\ell,s}}=0$, for all $s$ such that $ \mu_{j_m}(k)_s=0.$ Thus, we obtain  
	\begin{equation*}
		w_j(k)_m = \sum_{\ell_m\in I_{j_m}}c_{j_m}(k)_{m,\ell_m}\dfrac{g_{j_m}(k)_{\ell_m}}{\mu_{j_m}(k)_{\ell_m}},
	\end{equation*}
	where  $I_{j_m}=\{1\leq\ell_m\leq d_k \mid \mu_{j_m}(k)_{\ell_m}\neq 0\},$ and we denote  
	$c_{j_m}(k)_{m,\ell_m}=c_m(k)_{\ell_m}.$
		
	Following the same reasoning, for all $k\in\mathbb{N}_0\setminus\mathcal{Z}$ and $\ell=1,\dots,d_k$, we can express  
	\begin{equation}\label{ucgmu}
		\widehat{u}(k)_\ell = \sum_{m\in I_j}\overline{q_j(k)_{m,\ell}}\dfrac{g_j(k)_m}{\mu_j(k)_m}+\sum_{m\in Z_j}\overline{q_j(k)_{m,\ell}}\sum_{\ell_m\in I_{j_m}}c_m(k)_{\ell_m}\dfrac{g_{j_m}(k)_{\ell_m}}{\mu_{j_m}(k)_{\ell_m}},
	\end{equation}
	where  $I_j=\{1\leq\ell\leq d_k\,:\, \mu_j(k)_\ell\neq 0\},$ and $Z_j=\{1\leq\ell\leq d_k\,:\, \mu_j(k)_\ell= 0\}.$
	
	For each $k\in\mathbb{N}_0$, we define  
	\begin{equation*}
		\|\mu(k)\| = \min\left(\{|\mu_j(k)_\ell|\,:\, 1\leq j\leq n,\ 1\leq \ell\leq d_k\}\setminus\{0\}\right),
	\end{equation*}
	which represents the smallest nonzero eigenvalue (in norm) among the matrices $\sigma_j(k)$, for $j=1,\dots,n$. 
	
	We consider the following condition: there exists $\gamma>0$ such that  
	\begin{equation}\label{DC}
		\lambda_k\geq\gamma\ \Rightarrow\ \|\mu(k)\|\geq (1+\lambda_k)^{-\gamma}.
	\end{equation}
	
	\begin{theorem}\label{gh_system_normal}
		If $\mathcal{Z}$ is finite and \eqref{DC} holds, $\mathbb{P}$ is globally hypoelliptic.		
	\end{theorem}
	
	\begin{proof}
		Let $u\in\mathscr{D}'(M)$ and $f_j\in C^\infty(M)$ such that $P_ju=f_j$, for $j=1,\dots,n$. From (\ref{ucgmu}), for sufficiently large $k$, we obtain  
		\begin{equation*}
			|\widehat{u}(k)_\ell| \leq \dfrac{d_k}{\|\mu(k)\|}\max_{1\leq j\leq n}\|G_j(k)\| = \dfrac{d_k}{\|\mu(k)\|}\max_{1\leq j\leq n}\|\widehat{f_j}(k)\|.
		\end{equation*}
		Since $|q_j(k)_{m,\ell}|\leq 1$ and $|c_m(k)_{\ell_m}|\leq 1$ (as they are coefficients of unitary matrices), we have  
		\begin{equation*}
			|g_j(k)_m|\leq \|G_j(k)\|,\quad\forall m=1,\dots,d_k,\ \forall j=1,\dots,n.
		\end{equation*}
		Moreover, by (\ref{QGF}),  
		\begin{equation*}
			\|G_j(k)\|=\|Q_j(k)\widehat{f_j}(k)\|=\|\widehat{f_j}(k)\|,\quad \forall j=1,\dots,n.
		\end{equation*}
		
		Since $f_j\in C^\infty(M)$, for each $N\in\mathbb{N}$, there exists $C_N>0$ such that  
		\begin{equation*}
			\|\widehat{f_j}(k)\|\leq C_N(1+\lambda_k)^{-N-\gamma-\kappa},\quad\forall j=1,\dots,n,
		\end{equation*}
		where $\kappa=d/\nu$, with $d=\dim(M)$ and $\nu>0$ the order of the elliptic operator $E$. By (\ref{asymp}), there exists $K>0$ such that $d_k\leq K(1+\lambda_k)^{\kappa}$ for all $k\in\mathbb{N}_0$. Consequently,  
		\begin{align*}
			\|\widehat{u}(k)\|^2 & = \sum_{\ell=1}^{d_k}|\widehat{u}(k)_\ell|^2  
			\leq \dfrac{d_k^2}{\|\mu(k)\|^2}\max_{1\leq j\leq n}\|\widehat{f_j}(k)\|^2\\
			& \leq K^2(1+\lambda_k)^{2\kappa}C_N^2(1+\lambda_k)^{2\gamma}(1+\lambda_k)^{-2N-2\gamma-2\kappa}\\
			& = K^2C_N^2(1+\lambda_k)^{-2N}.
		\end{align*}
		This implies $\|\widehat{u}(k)\|\leq C(1+\lambda_k)^{-N}$, with $C=KC_N$. Thus, $u\in C^\infty(M)$, proving that $\mathbb{P}$ is globally hypoelliptic.	
	\end{proof}
	
	If we remove the finiteness assumption on $\mathcal{Z}$, by choosing $\widehat{u}(k)=0$ for $k\in\mathcal{Z}$, we obtain the following result for the global solvability of $\mathbb{P}$:
	
	\begin{theorem}\label{gs_system_normal}
		If \eqref{DC} holds, then $\mathbb{P}$ is globally solvable.
	\end{theorem}
	
	\begin{proof}
		Let $u\in\mathscr{D}'(M)$ and suppose that $P_ju=f_j\in C^\infty(M)$ for all $j=1,\dots,n$. For each $k\in\mathbb{N}_0$, we decompose $\widehat{u}(k) = \widehat{v}(k)+\widehat{w}(k),$ where $\widehat{v}(k)\in\ker\mathbb{P}(k)^\perp$ and $\widehat{w}(k)\in\ker\mathbb{P}(k)$. Then, the distribution $v$ with coefficients $\widehat{v}(k)$, satisfies $P_ju=P_jv$ for all $j=1,\dots,n$.
		
		To show that $v\in C^\infty(M)$, we consider, for each $j=1,\dots,n$ and $k\in\mathbb{N}_0$, the restriction $\tilde\sigma_j(k)$ of $\sigma_j(k)$ to $\ker\mathbb{P}(k)^\perp$, which satisfies  
		\begin{equation*}
			\bigcap_{j=1}^n\ker\tilde\sigma_j(k)=\{0\}.
		\end{equation*}
		
		Proceeding as in the globally hypoelliptic case, \eqref{DC} implies that  $v$ is smooth. Thus, $\mathbb{P}$ is almost globally hypoelliptic, and by Theorem \ref{agh_gs_system}, the system $\mathbb{P}$ is globally solvable.
	\end{proof}

\begin{remark}
	In the study of left-invariant vector fields on compact Lie groups, where normality naturally occurs, our results recover and generalize results previously found for operators on the torus and in non-commutative Lie groups. 
	
	In particular, every left-invariant vector field $X\in\mathfrak{g}$ is a skew-symmetric operator (therefore normal) and commutes with the elliptic Laplace-Beltrami operator $\Delta_G$ if $G$ is a compact Lie group with Lie algebra $\mathfrak{g}$. In \cite{KMR20, KMR20b}, spectral criteria have been used to characterize the global hypoellipticity and solvability of such vector fields. Additionally, as examined in \cite{RuzhTur} and covered in Section 6 of \cite{DR2018} and Section 4 of \cite{KM2020}, the Fourier analysis resulting from representation theory is closely related to the Fourier analysis induced by the positive elliptic operator $I-\Delta_G$. 
	
	Relevant examples of such systems naturally arise in the setting of Lie groups. For instance, let $G$ be a compact Lie group with Lie algebra $\mathfrak{g}$, and let $\mathfrak{h} \subset \mathfrak{g}$ be a Lie subalgebra. Since $\mathfrak{h}$ satisfies the Frobenius condition, it induces an involutive structure on $G$. If $X_1, \dots, X_n$ form a basis for $\mathfrak{h}$, then the system $(X_1, \dots, X_n)$ defines a family of strongly invariant normal operators on $G$ with respect to $I - \Delta_G$.  

	These considerations shows that our results in Sections 3 and 4 not only recover but also generalize previous results in the literature. Furthermore,  in this Section, we introduce new techniques that enable us to derive explicit solutions for systems of normal operators, reinforcing the connection between spectral analysis and global properties of differential operators.
\end{remark}

\section{Systems of Commuting Normal Operators}

Let $\mathbb{P}=(P_1,\dots,P_n)$ be a system of strongly invariant normal operators satisfying the commutativity condition $	[P_i,P_j] = P_iP_j - P_jP_i = 0$, $i,j=1,\dots,n$, which ensures that the matrices $\sigma_1(k),\dots,\sigma_n(k)$ are simultaneously diagonalizable for all $k\in\mathbb{N}_0$.  

Following the notation from the previous section, for each $k\in\mathbb{N}_0$, we denote by $\mathcal{B}_k=\{\varphi_\ell^{k} \mid 1\leq \ell \leq d_k\}$ an orthonormal basis, and we represent each diagonal matrix as  
\begin{equation*}
	\sigma_j(k)= \text{diag}(\mu_j(k)_1,\dots,\mu_j(k)_{d_k}), \quad j=1,\ldots, n.
\end{equation*}

For each $k\in\mathbb{N}_0$ and $\ell=1,\dots,d_k$, we define  
\begin{equation*}
	\|\mu(k)_\ell\|=\max\{|\mu_j(k)_\ell| \mid j=1,\dots,n\},
\end{equation*}
and consider the set  
\begin{equation*}
	\mathcal{Z}=\left\{k\in\mathbb{N}_0 \mid \exists \ell\in\{1,\dots,d_k\} \text{ such that } \|\mu(k)_\ell\|=0\right\}.
\end{equation*}

	\begin{theorem}
		Let $\mathbb{P}=(P_1,\dots,P_n)$ be a system of strongly invariant normal operators that commute pairwise. Then, $\mathbb{P}$ is globally hypoelliptic if and only if the set $\mathcal{Z}$ is finite and the following condition holds: there exist $C,\gamma>0$ such that
		\begin{equation}\label{cond_comm}
			\lambda_k\geq\gamma\ \Rightarrow\ \|\mu(k)_\ell\|\geq C(1+\lambda_k)^{-\gamma},\quad \forall\ell=1,\dots,d_k.
		\end{equation}
	\end{theorem}
	
	\begin{proof}
		If $\mathcal{Z}$ is finite and \eqref{cond_comm} holds, then Theorem \ref{gh_system_normal} ensures that $\mathbb{P}$ is globally hypoelliptic.
		
		Conversely, if $\mathcal{Z}$ is infinite, then by Proposition \ref{prop_Z}, $\mathbb{P}$ is not globally hypoelliptic. Suppose now that $\mathcal{Z}$ is finite, but condition \eqref{cond_comm} does not hold. Then, for each $N\in\mathbb{N}\setminus\mathcal{Z}$, there exist indices $k_N\in\mathbb{N}_0$, $j_N\in\{1,\dots,n\}$, and $\ell_N\in\{1,\dots,d_{k_N}\}$ such that
		\begin{equation*}
			0<\|\mu_{j_N}(k_N)_{\ell_N}\|<(1+\lambda_{k_N})^{-N}.
		\end{equation*}
		
		For each $k\in\mathbb{N}$, $j=1,\dots,n$, and $\ell=1,\dots,d_k$, define  
		\begin{equation*}
			\widehat{f_j}(k)_\ell=\begin{cases}
				\mu_{j_N}(k_N)_{\ell_N}, & \text{if } k=k_N,\, j=j_N,\, \ell=\ell_N \text{ for some } N\in\mathbb{N}\setminus\mathcal{Z},\\
				0, & \text{otherwise.}
			\end{cases}
		\end{equation*}
		It is clear that each function \( f_j \), defined by the coefficients above, is smooth. Additionally, the distribution \( u \in \mathscr{D}'(M)\setminus C^\infty(M) \) given by
		\begin{equation*}
			\widehat{u}(k)_\ell =\begin{cases}
				1, & \text{if } k=k_N,\, \ell=\ell_N \text{ for some } N\in\mathbb{N}\setminus\mathcal{Z},\\
				0, & \text{otherwise.}
			\end{cases}
		\end{equation*}
		satisfies \( P_j u = f_j \) for all \( j=1,\dots,n \). Thus, $\mathbb{P}$ is not globally hypoelliptic.
	\end{proof}
	
	\begin{theorem}
		Let $\mathbb{P}=(P_1,\dots,P_n)$ be a system of strongly invariant normal operators that commute pairwise. Then, $\mathbb{P}$ is globally solvable if and only if \eqref{cond_comm} holds.
	\end{theorem}
	
	\begin{proof}
		If condition \eqref{cond_comm} holds, then Theorem \ref{gs_system_normal} guarantees that $\mathbb{P}$ is globally solvable.
		
		Conversely, suppose that $\mathbb{P}$ is globally solvable. Then, the estimate from Theorem \ref{agh_gs_system}(a) holds. For each \( k\in\mathbb{N}_0 \) and \( \ell=1,\dots,d_k \), let \( j_0\in\{1,\dots,n\} \) be such that $|\mu_{j_0}(k)_\ell|=\|\mu(k)_\ell\|$,
		and take \( \varphi \) as an associated eigenvector with \( \|\varphi\|=1 \). Then,
		\begin{align*}
			\|\mu(k)_\ell\|^2 &= |\mu_{j_0}(k)_\ell|^2\|\varphi\|^2\ \geq\ n^{-2}\sum_{j=1}^n|\mu_j(k)_\ell|^2\|\varphi\|^2  = n^{-2}\sum_{j=1}^n\|\sigma_j(k)\varphi\|^2\\
			& \geq\ n^{-2}C^{2}(1+\lambda_k)^{-2\gamma}.
		\end{align*}
		Therefore, \eqref{cond_comm} holds.
	\end{proof}
	
\begin{remark}
	The case of commuting systems is significantly simpler and follows as a consequence of the previous cases. However, this setting encompasses several important examples, such as systems defined on the torus (see \cite{AKM19,AM2021,Be1999,BP1999_jmaa}) and systems of left-invariant operators on products of compact Lie groups (see \cite{DKP2024arxiv}).
\end{remark}

\bibliographystyle{amsplain}
\bibliography{references}
	
\end{document}